\documentclass[12pt,a4paper]{article}

\usepackage{graphicx,tikz}
\usepackage{latexsym}
\usepackage{amsmath}
\usepackage{amssymb}
\usepackage{array}
\usepackage{arydshln}

\usepackage{cite}
\usepackage{stmaryrd}
\usepackage{enumerate}

\usepackage{hyperref}

\usepackage{color}
\usepackage{lineno}
\usepackage{graphicx}
\usepackage{ae}
\usepackage{amsmath}
\usepackage{amssymb}
\usepackage{latexsym}
\usepackage{url}
\usepackage{epsfig}
\usepackage{mathrsfs}
\usepackage{amsfonts}
\usepackage{amsthm}
\usepackage{bm}

\newcommand{\Cay}{\mathrm{Cay}}

\newcommand{\Aut}{\mathrm{Aut}}

\newtheorem{theorem}{Theorem}[section]

\newtheorem{lemma}[theorem]{Lemma}

\newtheorem{cor}[theorem]{Corollary}

\theoremstyle{definition}

\numberwithin{equation}{section} 
\allowdisplaybreaks    

\def\qed{\hfill$\Box$\vspace{12pt}}

\long\def\delete#1{}

\usepackage{xcolor}
\usepackage[normalem]{ulem}

\usepackage{stfloats}

\begin{document}
	\title {Enumeration of Cayley graphs over a nonabelian group of order $8p$}

	\author{Bei Ye$^{a,b}$,~Xiaogang Liu$^{a,b,c,}$\thanks{Supported by the National Natural Science Foundation of China (No. 12371358) and the Guangdong Basic and Applied
			Basic Research Foundation (No. 2023A1515010986).}~$^,$\thanks{ Corresponding author. Email addresses: ybei@mail.nwpu.edu.cn, xiaogliu@nwpu.edu.cn, wj66@mail.nwpu.edu.cn},~Jing Wang$^{a,b,c}$
		\\[2mm]
		{\small $^a$School of Mathematics and Statistics,}\\[-0.8ex]
		{\small Northwestern Polytechnical University, Xi'an, Shaanxi 710072, P.R.~China}\\
		{\small $^b$Research \& Development Institute of Northwestern Polytechnical University in Shenzhen,}\\[-0.8ex]
		{\small Shenzhen, Guandong 518063, P.R. China}\\
		{\small $^c$Xi'an-Budapest Joint Research Center for Combinatorics,}\\[-0.8ex]
		{\small Northwestern Polytechnical University, Xi'an, Shaanxi 710129, P.R. China}\\
	}
	\date{}
	
	\openup 0.5\jot
	\maketitle
	\begin{abstract}
Let $T_{8p} = \left\langle a,b\mid a^{2p}=b^8=e,a^p=b^4,b^{-1}ab=a^{-1} \right\rangle$ be a nonabelian group of order $8p$, where $p$ is an odd prime number.
In this paper, we give the formula to calculate the number of Cayley graphs over $T_{8p}$ up to isomorphism by using the P\'olya Enumeration Theorem. Moreover, we get the formula to calculate the number of connected Cayley graphs over $T_{8p}$ by deleting the disconnected graphs. By applying the results, we list the exact number of (connected) Cayley graphs for $3\leq p \leq 13$.
		\smallskip
		
		

	\end{abstract}

	\section{Introduction}
Let $G$ be a finite group with the identity element $e$, and $\emptyset \neq S \subseteq G \backslash \{e\}$. The \emph{Cayley digraph} $ \Gamma = \Cay(G,S)$ is denoted to be the directed graph whose  vertex set is $G$ and two vertices $x, y \in G$ are adjacent if and only if $xy^{-1} \in S$. If $S =S^{-1}=\{s^{-1}\mid s\in S \}$ (called an \emph{inverse-closed subset}), then $\Cay(G,S)$ is an undirected  graph, which is called a \emph{Cayley graph}. In particular, if $G$ is a cyclic group, then $\Cay(G,S)$ is called a \emph{circulant (di)graph}. A Cayley digraph $\Cay(G,S)$ is called a \emph{DCI-graph} if for any Cayley digraph $\Cay(G,T)$, whenever $\Cay(G,T) \cong \Cay(G,S)$, there exists an automorphism $\sigma \in \Aut(G)$ such that $\sigma(S) = T$, where $\Aut(G)$ denotes the automorphism group of $G$. Similarly, if $\Cay(G,S)$ is a Cayley graph, and  for any Cayley graph $\Cay(G,T)$, whenever $\Cay(G,T) \cong \Cay(G,S)$, there exists an automorphism $\sigma \in \Aut(G)$ such that $\sigma(S) = T$, then $\Cay(G,S)$ is called a \emph{CI-graph}. A group $G$ is called a \emph{DCI-group} (respectively, \emph{CI-group}) if every Cayley digraph (respectively, Cayley graph) over $G$ is a DCI-graph (respectively, CI-graph).

In 1967, \'Ad\'am \cite{Adam1967} pioneered the study of DCI-graphs (DCI-groups) by proposing a conjecture: all circulant digraphs are DCI-graphs. Unfortunately, this conjecture was disproved by Elspas and Turner \cite{Elspas1970} in 1970. Despite its failure, this conjecture stimulated significant research in characterizing DCI-graphs (DCI-groups) and CI-graphs (CI-groups). For more information, please refer to \cite{Babai1977,Dobson1995,Dobson2015,Huang2000,Huang2003,CI-group2007,Muzychuk1995,Muzychuk1997}.
	
Determining the number of non-isomorphic Cayley digraphs (Cayley graphs)  is an interesting yet challenging problem, where DCI-groups (CI-groups) play an important role. In 1967, Turner \cite{Turner1967} found that P\'olya Enumeration Theorem is a suitable tool for counting  Cayley (di)graphs. Building on this, in 2000, Mishna \cite{Mishna2000} investigated the number of Cayley (di)graphs over cyclic groups, which was also studied in \cite{Alspach2002,Liskovets2000}. In 2017, Huang et al. \cite{HuangXY2017} considered the number of cubic Cayley graphs over dihedral groups. In 2019, Huang and Huang \cite{HuangXY2019} continued to study the number of Cayley (di)graphs over dihedral groups. In 2021, Wang, Liu and Feng \cite{WangLIU2021} studied the number of connected dicirculant graphs of order $4p$~($p$ prime). In 2024, Wang et al. \cite{WJing2024} calculated the number of (connected) dicirculant digraphs of order $4p$ ($p$ prime), and they got the number of (connected) dicirculant digraphs of out-degree $k$ for every $k$.
	
In this paper, we consider the enumeration of  Cayley graphs over the nonabelian group
	$$ T_{8p} = \left\langle a,b \mid a^{2p}=b^8=e,a^p=b^4,b^{-1}ab=a^{-1} \right\rangle, $$
where $p$ is an odd prime number, which comes from \cite{CI-group2007}. We give the formula to calculate the number of Cayley graphs over $T_{8p}$ by using the P\'olya Enumeration Theorem (See Theorem \ref{thm:graphnumber}). Moreover, we get the formula to calculate the number of connected Cayley graphs over $T_{8p}$ by deleting the disconnected graphs (See Theorem \ref{thm:connectgraphs}). Finally, we list the exact number of (connected) Cayley graphs for $3\leq p \leq 13$ (See Table \ref{table:graph}).

\section{Preliminaries}\label{Preliminaries}
	
	Let $G$ be a finite group with the identity element $e$ and let $X$ be a set. An \emph{action} of $G$ on $X$, denoted by $(G,X)$, is a map $\phi:G \times X \rightarrow X$ defined by $\phi(g,x) = gx$ such that
	\begin{itemize}
		\item [\rm $(1)$] $ex = x$, for all $x \in X$;
		\item [\rm $(2)$] $(g_{1}g_{2})x=g_{1}(g_{2}x)$, for all $x \in X$ and all $g_{1},g_{2} \in G$.
	\end{itemize}
	If there is an action of $G$ on $X$, then we say that $G$ \emph{acts on} $X$.  For each $x\in X$,
	the \emph{orbit} of $x$ is defined to be
	$$Gx=\{gx\mid g\in G \},$$
and the \emph{stabilizer} of $x$ is defined to be
	$$G_{x} = \{g\in G \mid gx = x \}.$$
	
	Let $H$ be a permutation group on $X~(|X|=n)$. Then $H$ acts on $X$ naturally by defining
$$
hx=h(x),~\forall~h\in H,~x\in X.
$$
It is known that every permutation $h\in H$ has a unique disjoint cycle decomposition. Let $b_{k}(h)$ be the number of the cycles of length $k$ in the cycle decomposition of a permutation $h$. The \emph{cycle type} of $h\in H$ is define as
	$$ \mathcal{T}(h) = (b_{1}(h),b_{2}(h),\dots,b_{n}(n)). $$
		The \emph{cycle index} $\mathcal{I}(H,X)$ of the action $(H,X)$ is defined as the following polynomial
		\begin{equation}\label{equ:cycleindex}
			\mathcal{I}(H,X) = P_{H}(x_{1},x_{2},\dots,x_{n}) = \frac{1}{|H|}\sum\limits_{h\in H}x_{1}^{b_{1}(h)}x_{2}^{b_{2}(h)}\cdots x_{n}^{b_{n}(h)}.
		\end{equation}
	
Let $D$ and $R$ be two finite sets, and denote by  $R^D$ the set of all mappings from $D$ to $R$. Let $H$ be a permutation group acting on $D$. Then we obtain a group action $(H,R^D)$ by
$$
\sigma f = f \circ \sigma^{-1}, ~~~~\forall~\sigma \in H,~ \forall~f \in R^D,
$$
where $f \circ \sigma^{-1}$ denotes the composite of two maps $f$ and $\sigma^{-1}$. Under the group action $(H,R^D)$, two mappings in $R^D$ are said to be \emph{$H$-equivalent} if they belong to the same orbit. The P\'olya Enumeration Theorem provides the number of orbits of the group action $(H,R^D)$.
	\begin{lemma}\emph{(P{\'o}lya Enumeration Theorem, see \cite[Chap. 2]{HararyP1973})} \label{Lem:Polya}
		Let $D$ and $R$ be finite sets with $|D|=n$ and $|R|=m$. Let $H$ be a permutation group acting on $D$. Denote by $\mathcal{F}$ the set of all orbits of the group action $(H, R^D)$. Then
		$$
		|\mathcal{F}|=P_H(m, m, \ldots, m),
		$$
		where $P_H(x_1, x_2, \ldots, x_n)$ is the cycle index of $(H,D)$ defined in \eqref{equ:cycleindex}.
	\end{lemma}

Let
$$
\left\{ a^{r},a^{r}b,a^{r}b^{2},a^{r}b^{3}\mid 0\leq r \leq 2p-1\right\}
$$
denote the set of all elements of the group $T_{8p}$. Then we obtain the following results by direct calculation.
	
	\begin{lemma}\label{Lem:computations}
		For the elements of $T_{8p}$, the following hold:
		\begin{itemize}
			\item[\rm (1)] $a^{r}b=ba^{-r},~a^{r}b^2=b^2a^{r}$;
			\item[\rm (2)] $(a^rb)^{-1}=a^{r+p}b^3,~ (a^r b^2)^{-1}=a^{p-r}b^2,~ (a^{r}b^{3})^{-1}=a^{r+p}b$.
		\end{itemize}
	\end{lemma}

	\begin{lemma}\emph{(See \cite[Theorem~1.3]{CI-group2007})}\label{Lem:CI}
		For any odd prime $p$, the group
		$$  G = \langle x,b \mid  x^p = 1, b^r = 1, b^{-1}xb = x^{-1} \rangle, ~~ \text{for~} r=4 \text{~or~} 8 $$
is a CI-group.
	\end{lemma}
	
	\begin{lemma}\label{Lem:CIT8p}
		For any odd prime $p$, the group
		$$ T_{8p} = \langle a,b\mid a^{2p}=b^8=1,a^p=b^4,b^{-1}ab=a^{-1} \rangle $$
		is a CI-group.
	\end{lemma}
\begin{proof}
By Lemma \ref{Lem:CI}, the group
	$$ G = \langle x,b \mid  x^p = 1, b^8 = 1, b^{-1}xb = x^{-1} \rangle$$
	is a CI-group, where $p$ is an odd prime. In the following, we show that $T_{8p}$ is isomorphic to $G$, and then $T_{8p}$ is a CI-group.
	
Recall that every element in $T_{8p}$ can be uniquely expressed as $a^kb^l$ with $0 \leq k < 2p$ and $0 \leq l < 4$. Then $\forall~a^kb^l, a^mb^n \in T_{8p}$, since $b^4 = a^p$, by Lemma \ref{Lem:computations}, we have
	\begin{eqnarray}\label{equ:CI1}
		(a^kb^l)(a^mb^n) = a^k(b^la^m)b^n =
		\begin{cases}
			a^{k+(-1)^lm}b^{l+n}, & \text{if } l+n < 4, \\
			a^{k+(-1)^lm+p}b^{l+n-4}, & \text{if } l+n \geq 4.
		\end{cases}
	\end{eqnarray}
Note that every element in $G$ can be uniquely expressed as $x^ib^j$ with $0 \leq i < p$ and $0 \leq j < 8$. Then $\forall~x^ib^j, x^mb^n \in G$, we have
	\begin{eqnarray}\label{equ:CI2}
	(x^ib^j)(x^mb^n) = x^i(b^jx^m)b^n = x^{i+(-1)^jm\pmod {p}}b^{j+n \pmod {8}}.
	\end{eqnarray}

Define a mapping $f: T_{8p} \to G$ such that
	\begin{equation*}
		f(a^kb^l) = \begin{cases}
			x^{\frac{k}{2}}b^l, & \text{if $k$ is even}, \\
			x^{\frac{k-p}{2}\pmod {p}}b^{l+4}, & \text{if $k$ is odd},
		\end{cases}
	\end{equation*}
for every $a^kb^l \in T_{8p}$.  So, $\forall ~x^ib^j \in G~(0\leq i < p,~0\leq j < 8)$, if $0\leq j < 4$, we choose $k = 2i~(0\leq k < 2p,\text{ $k$ is even})$ and $l=j~(0\leq l < 4)$, and then
$$
f(a^kb^l) = f(a^{2i}b^l)= x^ib^j;
$$
if $4\leq j < 8$, setting $j = l+4~(0\leq l < 4)$, we choose $k = 2i+p~(0\leq k <2p,\text{ $k$ is odd})$, and then
$$
f(a^{k}b^l)=f(a^{2i+p}b^l)=x^{i}b^{l+4}=x^ib^j.
$$
Thus $f$ is surjective.

Suppose that $f(a^kb^l) = f(a^mb^n)$, where $a^kb^l,a^mb^n \in T_{8p}$.
If $k$ and $m$ are both even, then
$$
f(a^kb^l)= x^{\frac{k}{2}}b^l = f(a^mb^n) = x^{\frac{m}{2}}b^n,
$$
which implies that $k=m$ and $l=n$, that is $a^kb^l=a^mb^n $. If $k$ and $m$ are both odd, then
$$
f(a^kb^l)= x^{\frac{k-p}{2}}b^{l+4} = f(a^mb^n) = x^{\frac{m-p}{2}}b^{n+4},
$$
which also implies that $a^kb^l = a^mb^n$. If $k$ is even and $m$ is odd, we have
$$ f(a^kb^l)= x^{\frac{k}{2}}b^l = f(a^mb^n) = x^{\frac{m-p}{2}}b^{n+4},$$
where $l<4$ and $n+4 \geq 4$, and this is impossible. Similarly, it is impossible if $k$ is odd and $m$ is even. Thus $f$ is injective.

Additionally, $\forall~a^kb^l,a^mb^n \in T_{8p}$, by \eqref{equ:CI1} and \eqref{equ:CI2}, one can easily verify that
$$ f(a^kb^la^mb^n) = f(a^kb^l)f(a^mb^n).$$
Thus $f$ is an isomorphism and $T_{8p} \cong G$.

Therefore, by Lemma \ref{Lem:CI}, $T_{8p}$ is a CI-group for any odd prime $p$.
\qed\end{proof}

Lemma \ref{Lem:CIT8p} implies the following result immediately.

	\begin{cor}\label{cor:S=T}
Let $p$ be an odd prime. Then two Cayley graphs $\Cay(T_{8p}, S)$ and $\Cay(T_{8p}, T)$ are isomorphic if and only if there exists $\sigma \in \Aut(T_{8p})$ such that $\sigma(S)=T$.
	\end{cor}
	\begin{lemma}\emph{(See \cite[Theorem 42]{Shanks1993})}\label{Lem:Z2Pgroup}
		Let $p$ be a prime and $\mathbb{Z}_{2p}^*$ be the multiplicative group of congruence classes modulo $2p$. Then $\mathbb{Z}_{2p}^*$ is cyclic.
	\end{lemma}

	By Lemma \ref{Lem:Z2Pgroup}, we see that $\mathbb{Z}_{2p}^{*} = \{ 1,3,\dots , p-2, p+2, \dots , 2p-1  \}$ is a cyclic group of order $\varphi(2p) = p-1 $, where $\varphi$ is the Euler's totient function. Consequently, we obtain the following result.

	\begin{lemma}\label{Lem:Z2Pcyclic}
		Assume that $\mathbb{Z}_{2p}^*=\langle z\rangle$ for some integer $z \in \mathbb{Z}_{2p}^*$. Then, for any $\alpha \in \mathbb{Z}_{2p}^*$, there exists $i_\alpha \in \mathbb{Z}_{p-1}$ such that $\alpha=z^{i_\alpha}$. Furthermore, if $\alpha$ ranges over all elements of $\mathbb{Z}_{2p}^*$, then $i_\alpha$ ranges over all elements of $\mathbb{Z}_{p-1}$.
	\end{lemma}
	\begin{lemma}\label{Lem:Aut}
Let $p$ be an odd prime. Then the automorphism group of $T_{8p}$ is
		$$ \Aut(T_{8p}) = \left\{\sigma_{\alpha,\beta},\tau_{\gamma,\delta}\mid \alpha,\gamma \in \mathbb{Z}_{2p}^{*}, \beta,\delta \in \mathbb{Z}_{2p} \right\}, $$
		where $ \sigma_{\alpha,\beta}(a) = a^{\alpha} $, $\sigma_{\alpha,\beta}(b) = a^{\beta}b $ and $\tau_{\gamma,\delta}(a) = a^{\gamma}$, $\tau_{\gamma,\delta}(b) = a^{\delta}b^3$.
	\end{lemma}
	\begin{proof}
Note that Table \ref{Order} gives the orders of elements in the group $T_{8p}$, where $\gcd(x,y)$ denotes the greatest common divisor of $x$ and $y$, and $\mathrm{lcm}(x,y)$ denotes the least common multiple of $x$ and $y$.
		\begin{table}[ht]
			\caption{\textbf{ Orders of Elements in $T_{8p}$ }}
			\label{Order}
			\centering
			\begin{tabular}{c|c|c}
				\hline
				Type & Elements & Orders   \\
				\hline
				I & $a^i$ & $\frac{2p}{\gcd(i,2p)}$  \\
				\hline
				II & $ a^ib^j $~($j$ is even) & $ \mathrm{lcm}\left(\frac{2p}{\gcd(i,2p)}, 4\right)$ \\
				\hline
				III & $ a^ib^j $~($j$ is odd)  &  8  \\
				\hline
			\end{tabular}
		\end{table}
		
Suppose that $\sigma_{\alpha,\beta}$ and $\tau_{\gamma,\delta}$ are two automorphisms of $T_{8p}$. Since  every  automorphism of $T_{8p}$ is uniquely determined by its action on the generators $a$ and $b$, the elements $\sigma_{\alpha,\beta}(a),\tau_{\gamma,\delta}(a) $ and $\sigma_{\alpha,\beta}(b),\tau_{\gamma,\delta}(b)$ must  have orders $2p$ and $8$, respectively, because  $o(a)=2p$ and $o(b) = 8$. Note that $2p \not\equiv 0 \pmod {4}$ when $p$ is an odd prime. By Table \ref{Order}, we know that elements of order $2p$ in $T_{8p}$ are $ a^i$ satisfying $\gcd(i,2p)=1$,  and the elements of order $8$ in $T_{8p}$ are  either $a^ib$ or $a^ib^3$ for $0\leq i \leq 2p-1$. Then one can easily verify that the automorphism group
		$\Aut(T_{8p})$ consists of two types of mappings:
		\begin{equation*}
			\begin{cases}
				\sigma_{\alpha,\beta}(a) = a^{\alpha}, \\
				\sigma_{\alpha,\beta}(b) = a^{\beta}b,
			\end{cases} \text{and} \quad
			\begin{cases}
				\tau_{\gamma,\delta}(a) = a^{\gamma}, \\
				\tau_{\gamma,\delta}(b) = a^{\delta}b^3,
			\end{cases}
		\end{equation*} where $ \alpha,\gamma \in \mathbb{Z}_{2p}^{*}$, $\beta,\delta \in \mathbb{Z}_{2p}$, which are equivalent to
		\begin{equation}\label{equ:Aut}
			\begin{cases}
				\sigma_{\alpha,\beta}(a^i) = a^{i\alpha}, \\
				\sigma_{\alpha,\beta}(a^ib) = a^{i\alpha+\beta}b, \\
				\sigma_{\alpha,\beta}(a^ib^2) = a^{i\alpha}b^2, \\
				\sigma_{\alpha,\beta}(a^ib^3) = a^{i\alpha+\beta}b^3,
			\end{cases} \text{and} \quad
			\begin{cases}
				\tau_{\gamma,\delta}(a^i) = a^{i\gamma}, \\
				\tau_{\gamma,\delta}(a^ib) = a^{i\gamma+\delta}b^3, \\
				\tau_{\gamma,\delta}(a^ib^2) = a^{i\gamma+p}b^2, \\
				\tau_{\gamma,\delta}(a^ib^3) = a^{i\gamma+\delta}b,
			\end{cases}
		\end{equation}
where $i \in \mathbb{Z}_{2p}$. This completes the proof.
\qed\end{proof}


	\begin{lemma}\emph{(See \cite[Lemma~3.1]{WJing2024})}\label{lem:wj}
		Let $p$ be an odd prime. Fix an element $1\neq \alpha \in \mathbb{Z}_{2p}^*$. For any $t\in 2\mathbb{Z}_{2p}\backslash\{0\}=\{2,4,6,\ldots,2p-2\}$, there exists a unique $x\in \mathbb{Z}_{2p}^*$ such that
		$x-\alpha x=t $ in $2\mathbb{Z}_{2p}$.
	\end{lemma}
	

\section{Enumerating Cayley graphs over the group $T_{8p}$} \label{graph}
	
Let
$$
D = \left\{ \{a^i,(a^i)^{-1} \},\ldots , \{a^p\},  \{a^jb,(a^jb)^{-1}\},\ldots,\{a^jb^2,(a^jb^2)^{-1}\},\ldots\right\},
$$
where $i\in \mathbb{Z}_{2p}\backslash \{0\},j\in \mathbb{Z}_{2p}$, and
$$
R=\{0,1\}.
$$
For $S \subseteq D $, let $f_{S}$ denote the characteristic function of $S$, that is,
	\begin{equation*}
		f_{S}\left(\{s,s^{-1}\}\right) = \left\{
		\begin{array}{cc}
			1, & \left\{s,s^{-1}\right\} \subseteq S,  \\
			0, & \text{otherwise}.
		\end{array}
		\right.
	\end{equation*}
Clearly, $f_{S}\in R^D= \{f\mid f:D \to R\}$, and $R^D$ consists of all characteristic functions on $D$. Note that $ \Aut(T_{8p})$ is a permutation group acting on $D$ and $R^D$. Note also Corollary \ref{cor:S=T} that two Cayley graphs $\Cay(T_{8p},S)$ and $\Cay(T_{8p},T)$ over $T_{8p}$ are isomorphic if and only if there exists an automorphism $\sigma \in \Aut(T_{8p})$ such that $\sigma(S)=T$, that is, characteristic functions $f_S, f_T \in R^D$ are $\Aut(T_{8p})$-equivalent. Therefore, the number of Cayley graphs on $T_{8p}$ up to isomorphism is equal to the number of orbits of the group action $(\Aut(T_{8p}),R^D)$. By Lemma \ref{Lem:Polya}, we should compute the cycle index $\mathcal{I}(\Aut(T_{8p}),D)$ of $\Aut(T_{8p})$ acting on $D$ for enumerating Cayley graphs over $T_{8p}$.
	
	For convenience, we set that $D=A\cup B$, where $A=A_{1}\cup A_{2} \cup \{a^p\}$, and
	$$ A_{1} = \{\bar{a}^{i}\mid 1\leq i \leq p-1 \},~~\text{where}~\bar{a}^i = \{a^i,a^{-i}\},$$
	$$ A_{2} = \left\{\bar{a}^lb^2 \mid 0\leq l \leq \frac{p-1}{2},~p+1\leq l \leq \frac{3p-1}{2}\right\},~~\text{where}~\bar{a}^lb^2=\{a^lb^2,a^{p-l}b^2\},$$
	$$ B=\{\bar{a}^jb\mid j \in \mathbb{Z}_{2p}\},~~\text{where}~\bar{a}^jb = \{a^jb,a^{p+j}b^3\}.$$
	
We obtain the cycle index $\mathcal{I}(\Aut(T_{8p}),D)$ as follows.

	\begin{lemma}\label{lem:graphcycletype}
		Let $p$ be an odd prime. Let $\sigma_{\alpha,1},\tau_{\gamma,\delta} \in \Aut(T_{8p})$ and $\mathbb{Z}_{2p}^* = \langle z \rangle $ as defined in Lemmas \ref{Lem:Z2Pcyclic} and \ref{Lem:Aut}. Under the action of $\Aut(T_{8p})$ on $D$, the cycle types of $\sigma_{\alpha,\beta}$, $\tau_{\gamma,\delta}$ are given by $\mathcal{T}(\sigma_{\alpha,\beta}) = (b_{1}(\sigma_{\alpha,\beta}),b_{2}(\sigma_{\alpha,\beta}), \ldots,b_{4p}(\sigma_{\alpha,\beta}))$ and $\mathcal{T}(\tau_{\gamma,\delta}) = (b_{1}(\tau_{\gamma,\delta}),b_{2}(\tau_{\gamma,\delta}), \ldots,b_{4p}(\tau_{\gamma,\delta}))$, respectively, where
		\begin{equation}\label{equ2:1bk1}
			b_k(\sigma_{1,\beta})=\left\{
			\begin{array}{ll}
				4p,      &\text{if}~ k=1  ~\text{and}~ \beta=0,\\[0.2cm]
				2p,       &\text{if}~ k=1  ~\text{and}~ \beta \in \mathbb{Z}_{2p}\backslash \{0\},\\[0.2cm]
				p,          &\text{if}~ k=2  ~\text{and}~ \beta =p,\\[0.2cm]
				2,         		&\text{if}~ k=p  ~\text {and}~ \beta \in 2\mathbb{Z}_{2p} \backslash\{0\},\\[0.2cm]
				1,          &\text{if}~ k=2p ~\text {and}~ \beta \in (2\mathbb{Z}_{2p}+1) \backslash\{p\},\\[0.2cm]
				0,          &\text{otherwise};
			\end{array}\right.
		\end{equation}
		\begin{equation}\label{equ2:2bk1}
			b_k(\tau_{1,\delta})=\left\{
			\begin{array}{ll}
				p+1,      &\text{if}~ k=1  ~\text {and}~ \delta\in \mathbb{Z}_{2p}\backslash \{p\},\\[0.2cm]
				3p+1,          &\text{if}~ k=1  ~\text {and}~ \delta =p,\\[0.2cm]
				\frac{p-1}{2},          &\text{if}~ k=2  ~\text {and}~ \delta \in \mathbb{Z}_{2p}\backslash\{0\},\\[0.2cm]
				\frac{3p-1}{2},       &\text{if}~ k=2  ~\text {and}~ \delta = 0,\\[0.2cm]
				2,          &\text{if}~ k=p ~\text {and}~ \delta \in 2\mathbb{Z}_{2p}+1 \backslash\{p\},\\[0.2cm]
				1,          &\text{if}~ k=2p ~\text {and}~ \delta \in 2\mathbb{Z}_{2p} \backslash\{0\},\\[0.2cm]
				0,          &\text{otherwise};
			\end{array}\right.
		\end{equation}
		\begin{equation}\label{equ2:1bkeven0}
			b_k(\sigma_{\alpha,\beta})=b_k(\sigma_{\alpha,0})=\left\{
			\begin{array}{ll}
				4,               & \text{if}~ k=1,\\[0.2cm]
				4\gcd(i_\alpha, p-1), & \text{if}~ k=\frac{p-1}{\gcd(i_\alpha, p-1)},\\[0.2cm]
				0,               & \text{otherwise},
			\end{array}\right.
		\end{equation}
where $1 \neq \alpha =z^{i_\alpha} \in \mathbb{Z}_{2p}^*$ (i.e., $i_\alpha\neq 0$) and $\beta \in 2\mathbb{Z}_{2p}$;
		\begin{equation}\label{equ2:1bkodd1}
			b_k(\sigma_{\alpha,\beta})=b_k(\sigma_{\alpha,1})=\left\{
			\begin{array}{ll}
				2,               & \text{if}~ k=1,\\[0.2cm]
				1,               & \text{if}~ k=2,\\[0.2cm]
				4\gcd(i_\alpha, p-1), & \text{if}~ k=\frac{p-1}{\gcd(i_\alpha, p-1)},~\text{and}~ \frac{p-1}{\gcd(i_\alpha, p-1)} ~\text{is~even},\\[0.2cm]
				2\gcd(i_\alpha, p-1), & \text{if}~ k=\frac{p-1}{\gcd(i_\alpha, p-1)},~\text{and}~ \frac{p-1}{\gcd(i_\alpha, p-1)} ~\text{is~odd},\\[0.2cm]
				\gcd(i_\alpha, p-1), & \text{if}~ k=\frac{2p-2}{\gcd(i_\alpha, p-1)},~\text{and}~ \frac{p-1}{\gcd(i_\alpha, p-1)} ~\text{is~odd},\\[0.2cm]
				0,               & \text{otherwise},
			\end{array}\right.
		\end{equation}
where $1 \neq \alpha =z^{i_\alpha} \in \mathbb{Z}_{2p}^*$ (i.e., $i_\alpha \neq 0$) and $\beta \in 2\mathbb{Z}_{2p}+1$;
		\begin{equation}\label{equ2:2bkeven0}
			b_{k}(\tau_{\gamma,\delta}) = b_{k}(\tau_{\gamma,0})=\left\{
			\begin{array}{ll}
				2,       & \text{if}~k=1,\\[0.2cm]
				1,       & \text{if}~k=2,\\[0.2cm]
				4\gcd(i_{\gamma},p-1),  & \text{if}~ k=\frac{p-1}{\gcd(i_\gamma, p-1)},~\text{and}~ \frac{p-1}{\gcd(i_\gamma, p-1)} ~\text{is~even},\\[0.2cm]
				\gcd(i_\gamma, p-1), & \text{if}~ k=\frac{p-1}{\gcd(i_\gamma, p-1)},~\text{and}~ \frac{p-1}{\gcd(i_\gamma, p-1)} ~\text{is~odd},\\[0.2cm]
				\frac{3\gcd(i_\gamma, p-1)}{2}, & \text{if}~ k=\frac{2p-2}{\gcd(i_\gamma, p-1)},~\text{and}~ \frac{p-1}{\gcd(i_\gamma, p-1)} ~\text{is~odd},\\[0.2cm]
				0,               & \text{otherwise},
			\end{array}
			\right.
		\end{equation}
where $1 \neq \gamma =z^{i_\gamma} \in \mathbb{Z}_{2p}^*$ (i.e., $i_\gamma \neq 0$) and $\delta \in 2\mathbb{Z}_{2p}$; and
		\begin{equation}\label{equ2:2bkodd1}
			b_{k}(\tau_{\gamma,\delta}) = b_{k}(\tau_{\gamma,1})=\left\{
			\begin{array}{ll}
				4,       & \text{if}~k=1,\\[0.2cm]
				4\gcd(i_{\gamma},p-1),  & \text{if}~ k=\frac{p-1}{\gcd(i_\gamma, p-1)},~\text{and}~ \frac{p-1}{\gcd(i_\gamma, p-1)} ~\text{is~even},\\[0.2cm]
				3\gcd(i_\gamma, p-1), & \text{if}~ k=\frac{p-1}{\gcd(i_\gamma, p-1)},~\text{and}~ \frac{p-1}{\gcd(i_\gamma, p-1)} ~\text{is~odd},\\[0.2cm]
				\frac{\gcd(i_\gamma, p-1)}{2}, & \text{if}~ k=\frac{2p-2}{\gcd(i_\gamma, p-1)},~\text{and}~ \frac{p-1}{\gcd(i_\gamma, p-1)} ~\text{is~odd},\\[0.2cm]
				0,               & \text{otherwise},
			\end{array}
			\right.
		\end{equation}
where $1 \neq \gamma =z^{i_\gamma} \in \mathbb{Z}_{2p}^*$ (i.e., $i_\gamma \neq 0$) and $\delta \in 2\mathbb{Z}_{2p}+1$.
	\end{lemma}
	\begin{proof}
By Lemma \ref{Lem:Aut}, for every $i \in \mathbb{Z}_{p}\backslash\{0\}$, $ l \in \{0,1,\ldots,\frac{p-1}{2},p+1,\ldots,\frac{3p-1}{2} \}$ and $j\in \mathbb{Z}_{2p}$, we have
		$$\sigma_{\alpha,\beta}(\bar{a}^i)=\bar{a}^{\alpha i},~~\sigma_{\alpha,\beta}(\bar{a}^lb^2) = \bar{a}^{\alpha l}b^2,~~\sigma_{\alpha,\beta}(a^p) = a^p,~~\sigma_{\alpha,\beta}(\bar{a}^jb) = \bar{a}^{\alpha j + \beta}b,$$
		$$\tau_{\gamma,\delta}(\bar{a}^i) = \bar{a}^{\gamma i},~~\tau_{\gamma,\delta}(\bar{a}^lb^2) = \bar{a}^{\gamma l+p}b^2,~~\tau_{\gamma,\delta}(a^p) = a^p,~~\tau_{\gamma,\delta}(\bar{a}^jb)=\bar{a}^{\gamma j+\delta+p}b,$$
		i.e.,
		$$ \sigma_{\alpha,\beta}(A_{1})=A_{1},~~\sigma_{\alpha,\beta}(A_{2}) = A_{2},~~\sigma_{\alpha,\beta}(a^p) = a^p,~~\sigma_{\alpha,\beta}(B) = B,$$
		$$ \tau_{\gamma,\delta}(A_{1}) = A_{1},~~\tau_{\gamma,\delta}(A_{2}) = A_{2},~~\tau_{\gamma,\delta}(a^p) = a^p,~~\tau_{\gamma,\delta}(B) = B.$$
		Hence,
		$$ b_{2p+1}(\sigma_{\alpha,\beta}) = b_{2p+2}(\sigma_{\alpha,\beta})=\cdots = b_{4p}(\sigma_{\alpha,\beta}) = 0,$$
		$$ b_{2p+1}(\tau_{\gamma,\delta}) = b_{2p+2}(\tau_{\gamma,\delta})=\cdots=b_{4p}(\tau_{\gamma,\delta}) = 0.$$
		Next, for $\sigma_{\alpha,\beta},\tau_{\gamma,\delta} \in \Aut(T_{8p})$, consider the following cases.

\noindent\emph{Case 1.} $\alpha=1$ and $\gamma = 1$.
	
\noindent\emph{Case 1.1.} $\alpha = 1$. In this case, we study the cycle type of $\sigma_{1,\beta}$ on $D=A\cup B$.
	
Note that $\sigma_{1,\beta}(a^p)=a^p$,  $\sigma_{1,\beta}(\bar{a}^i)=\bar{a}^i$ for each $i\in \mathbb{Z}_{p}\backslash \{0\}$ and $\sigma_{1,\beta}(\bar{a}^lb^2) = \bar{a}^lb^2$ for each $\bar{a}^lb^2\in A_{2}$. Then $\sigma_{1,\beta}$ splits $A$ into $1+(p-1)+p=2p$ cycles of length $1$.

Observe that $\sigma_{1,\beta}(\bar{a}^jb) = \bar{a}^{j+\beta}b$ for each $j \in \mathbb{Z}_{2p}$. If $\beta = 0$, then $\sigma_{1,0}(\bar{a}^jb) = \bar{a}^jb$ for each $j \in \mathbb{Z}_{2p}$. Thus $\sigma_{1,0}$ splits $B$ into $2p$ cycles of length $1$.
	
	If $\beta \in \mathbb{Z}_{2p}\backslash \{0\}$, the order of $\beta$ is
	\begin{equation*}
		o(\beta)=\frac{2p}{\gcd(\beta, 2p)}=\left\{
		\begin{array}{ll}
			2,     &\beta=p,\\[0.2cm]
			p,     &\beta \in 2\mathbb{Z}_{2p} \backslash\{0\},\\[0.2cm]
			2p,    &\beta \in (2\mathbb{Z}_{2p}+1) \backslash\{p\}.
		\end{array}\right.
	\end{equation*}
So, if $\beta = p$, then $\bar{a}^jb \in B$ is in the cycle
	$$(\bar{a}^jb,\sigma_{1,p}(\bar{a}^jb)) = (\bar{a}^jb,\bar{a}^{j+p}b)$$
	for any $j \in \mathbb{Z}_{2p}$. Thus $\sigma_{1,p}$ splits $B$ into $p$ cycles of length $2$. If $\beta \in  2\mathbb{Z}_{2p} \backslash\{0\}$, then $\bar{a}^jb\in B$ is in the cycle
	$$ \left(\bar{a}^jb,\sigma_{1,\beta}(\bar{a}^jb),\sigma_{1,\beta}^2(\bar{a}^jb),\ldots,\sigma_{1,\beta}^{p-1}(\bar{a}^jb)  \right) = \left(\bar{a}^jb,\bar{a}^{j+\beta}b,\bar{a}^{j+2\beta}b,\ldots,\bar{a}^{j+(p-1)\beta}b\right)
	$$
	for any $j\in \mathbb{Z}_{2p}$. Thus $\sigma_{1,\beta}$ splits $B$ into $2$ cycles of length $p$.
	If $\beta \in (2\mathbb{Z}_{2p}+1) \backslash\{p\}$, then $\bar{a}^jb\in B$ is in the cycle
	$$ \left(\bar{a}^jb,\sigma_{1,\beta}(\bar{a}^jb),\sigma_{1,\beta}^2(\bar{a}^jb),\ldots,\sigma_{1,\beta}^{2p-1}(\bar{a}^jb)  \right) = \left(\bar{a}^jb,\bar{a}^{j+\beta}b,\bar{a}^{j+2\beta}b,\ldots,\bar{a}^{j+(2p-1)\beta}b\right),
	$$
	for any $j\in \mathbb{Z}_{2p}$. Thus, the permutation $\sigma_{1,\beta}$ splits $B$ into $1$ cycle of length $2p$.
	
Therefore, we obtain the cycle type of $\sigma_{1,\beta}$, as shown in \eqref{equ2:1bk1}.
	
	\noindent\emph{Case 1.2.} $\gamma = 1$. In this case, we consider the cycle type of $\tau_{1,\delta}$ on $D=A\cup B$.
	
	If $\bar{a}^i \in A_{1}$ for any $ i \in \mathbb{Z}_{p}\backslash \{0\}$, then $\tau_{1,\delta}(\bar{a}^i )=\bar{a}^i $. Thus $\tau_{1,\delta}$ splits $A_{1}\cup\{a^p\}$ into $p$ cycles of length $1$. Notice that $\tau_{1,\delta}(b^2)=a^pb^2$, $\tau_{1,\delta}(a^pb^2) = b^2$ and $\bar{a}^0b^2=\{b^2,a^{p}b^2\}$. So $\tau_{1,\delta}(\bar{a}^0b^2)=\bar{a}^0b^2$. Note also that
	$$\tau_{1,\delta}(\bar{a}^lb^2)=\bar{a}^{l+p}b^2~~\text{and}~~ \tau_{1,\delta}^2(\bar{a}^lb^2) = \tau_{1,\delta}(\bar{a}^{l+p}b^2) = \bar{a}^lb^2$$ for every $\bar{a}^lb^2\in A_{2}\backslash\{\bar{a}^0b^2\}$. Thus $\tau_{1,\delta}$ splits $A_{2}$ into $\frac{p-1}{2}$ cycles of length $2$ and $1$ cycle of length $1$.
	
	Notice that $\tau_{1,\delta}(\bar{a}^jb)  = \bar{a}^{j+\delta+p}b$ for $j\in \mathbb{Z}_{2p}$. If $\delta=0$, then
	$$\tau_{1,0}^2(\bar{a}^jb) = \tau_{1,0}(\bar{a}^{j+p}b) = \bar{a}^jb.$$ Thus $\tau_{1,0}$ splits $B$ into $p$ cycles of length $2$. If $ \delta \in \mathbb{Z}_{2p} \backslash \{0\}$, then $\bar{a}^jb\in B$ is in the cycle
	$$ \left( \bar{a}^jb,\tau_{1,\delta}(\bar{a}^jb),\tau_{1,\delta}^2(\bar{a}^jb),\ldots,\tau_{1,\delta}^{m-1}(\bar{a}^jb)   \right) = \left(\bar{a}^jb,\bar{a}^{j+\delta+p}b,\bar{a}^{j+2\delta+2p}b,\ldots,\bar{a}^{j+(m-1)\delta+(m-1)p}b\right),$$
	where $m$ is the least positive integer such that
	$$ j + m\delta+mp \equiv j \pmod{2p},$$
	that is,
	$$ m(\delta+p) \equiv 0 \pmod {2p}.$$
	If $\delta = p$, then $m=1$, and then $\tau_{1,\delta}$ splits $B$ into $2p$ cycles of length $1$. If $\delta \in 2\mathbb{Z}_{2p}\backslash\{0\}$, then $m=2p$, and then $\tau_{1,\delta}$ splits $B$ into $1$ cycle of length $2p$. If $\delta \in 2\mathbb{Z}_{2p}+1\backslash\{p\}$, then $m=p$, and then $\tau_{1,\delta}$ splits $B$ into $2$ cycles of length $p$.
	
Therefore, we obtain the cycle type of $\tau_{1,\delta}$, as shown in \eqref{equ2:2bk1}.

	\noindent\emph{Case 2.} $1 \neq \alpha = z^{i_{\alpha}}$ and $1 \neq \gamma = z^{i_{\gamma}}$.

	By Lemma \ref{Lem:Z2Pcyclic}, the orders of $\alpha, \gamma \in \mathbb{Z}_{2p}^*$ are
	$$ o(\alpha) = o(z^{i_{\alpha}}) = \frac{p-1}{\gcd(i_{\alpha},p-1)}~~\text{and}~~o(\gamma) = o(z^{i_{\gamma}})=\frac{p-1}{\gcd(i_{\gamma},p-1)}.$$
	
	Next, we prove that
$$\mathcal{T}(\sigma_{\alpha,\beta})=\mathcal{T}(\sigma_{\alpha,0}),~ \mathcal{T}(\sigma_{\alpha,\beta+1})=\mathcal{T}(\sigma_{\alpha,1}),~~ \forall~\beta \in 2\mathbb{Z}_{2p} \backslash \{0\}, $$
$$\mathcal{T}(\tau_{\gamma,\delta})=\mathcal{T}(\tau_{\gamma,0}),~ \mathcal{T}(\tau_{\gamma,\delta+1})=\mathcal{T}(\tau_{\gamma,1}),~~ \forall~\delta \in 2\mathbb{Z}_{2p} \backslash \{0\}.$$
	
Recall that
	$$ \sigma_{\alpha,\beta}(\bar{a}^i) = \bar{a}^{\alpha i},~~\sigma_{\alpha,\beta}(\bar{a}^lb^2) =\bar{a}^{\alpha l}b^2,~~\sigma_{\alpha,\beta}(a^p)= a^p,~~\forall~\beta \in \mathbb{Z}_{2p}, $$
	$$ \tau_{\gamma,\delta}(\bar{a}^i)  = \bar{a}^{\gamma i},~~\tau_{\gamma,\delta}(\bar{a}^lb^2) =  \bar{a}^{\gamma l+p}b^2,~~\tau_{\gamma,\delta}(a^p)= a^p, ~~\forall~\delta \in \mathbb{Z}_{2p}.$$
Then all $\sigma_{\alpha,\beta}$'s (respectively, $\tau_{\gamma,\delta}$'s) have the same cycle type in $A$.

Consider the cycle type of $\sigma_{\alpha,\beta}$ in $B$. Let $\beta \in 2\mathbb{Z}_{2p}\backslash\{0\}$. By Lemma \ref{lem:wj}, there exists a unique $x \in \mathbb{Z}_{2p}^*$ such that $x-\alpha x = \beta$. Define a bijection $\phi$ in $B$ such that
$$
\phi(\bar{a}^jb) = \bar{a}^{j+x}b,~j \in \mathbb{Z}_{2p},
$$
where
	$$ \phi(a^jb) = a^{j+x}b,~~\phi(a^{p+j}b^3)=a^{(p+j)+x}b^3,~j \in \mathbb{Z}_{2p}.$$

	 Assume that
$$
(\bar{a}^{j_{0}}b,\bar{a}^{j_{1}}b,\ldots,\bar{a}^{j_{r-1}}b)~~(j_{0},j_{1},\ldots,j_{r-1} \in \mathbb{Z}_{2p})
$$
	is a cycle of $\sigma_{\alpha,0}$ in $B$. Then $j_{l} = \alpha j_{l-1}$ for  $l=1,2,\ldots,r-1$. Then we have
	$$\sigma_{\alpha,\beta}(\phi(\bar{a}^{j_{l-1}}b)) = \sigma_{\alpha,\beta}(\bar{a}^{j_{l-1}+x}b) = \bar{a}^{\alpha j_{l-1} + \alpha x + \beta}b = \bar{a}^{j_{l}+x}b = \phi(\bar{a}^{j_{l}}b).$$
	Thus $(\phi(\bar{a}^{j_{0}}b),\phi(\bar{a}^{j_{1}}b),\ldots,\phi(\bar{a}^{j_{r-1}}b))$ is a cycle of $\sigma_{\alpha,\beta}$ in $B$. Therefore, $\sigma_{\alpha,\beta}$ and $\sigma_{\alpha,0}$ have the same cycle type in $B$.
	
	 Assume that
$$
(\bar{a}^{j_{0}}b,\bar{a}^{j_{1}}b,\ldots,\bar{a}^{j_{r-1}}b)~~(j_{0},j_{1},\ldots,j_{r-1} \in \mathbb{Z}_{2p})
$$
is  a cycle of $\sigma_{\alpha,1}$ in $B$. Then $j_{l} = \alpha j_{l-1}+1$ for  $l=1,2,\ldots,r-1$.  Then we have
	$$\sigma_{\alpha,\beta+1}(\phi(\bar{a}^{j_{l-1}}b)) = \sigma_{\alpha,\beta+1}(\bar{a}^{j_{l-1}+x}b) = \bar{a}^{\alpha j_{l-1} + \alpha x + \beta+1}b = \bar{a}^{j_{l}+x}b = \phi(\bar{a}^{j_{l}}b).$$
	Thus $(\phi(\bar{a}^{j_{0}}b),\phi(\bar{a}^{j_{1}}b),\ldots,\phi(\bar{a}^{j_{r-1}}b))$
	is a cycle of $\sigma_{\alpha,\beta+1}$ in $B$. Therefore, $\sigma_{\alpha,\beta+1}$ and $\sigma_{\alpha,1}$ have the same cycle type  in $B$.
	
	Similarly, if $\delta \in 2\mathbb{Z}_{2p}\backslash \{0\}$, one can verify that $\tau_{\gamma,\delta}$ and $\tau_{\gamma,0}$ have the same cycle type in $B$, and $\tau_{\gamma,\delta+1}$ and $\tau_{\gamma,1}$ have the same cycle type in $B$, respectively.

In the following, we study the cycle types of $\sigma_{\alpha,0}$, $\sigma_{\alpha,1}$, $\tau_{\gamma,0}$ and $\tau_{\gamma,1}$, respectively.

	\noindent\emph{Case 2.1.} In this case, we consider the cycle type of $\sigma_{\alpha,0}$ on $D=A\cup B$.
	
	Notice that
	$$\sigma_{\alpha,0}(a^p)=a^p,~~\sigma_{\alpha,0}(\bar{a}^0b^2)=\bar{a}^0b^2, ~~\sigma_{\alpha,0}(\bar{a}^0b)=\bar{a}^0b,~~\sigma_{\alpha,0}(\bar{a}^pb)=\bar{a}^pb.$$
Note also that $\bar{a}^i\in A_{1}$ is in the cycle
	$$ (\bar{a}^i,\sigma_{\alpha,0}(\bar{a}^i),\sigma_{\alpha,0}^2(\bar{a}^i),\ldots,\sigma_{\alpha,0}^{o(\alpha)-1}(\bar{a}^i)) = (\bar{a}^i,\bar{a}^{\alpha i},\bar{a}^{\alpha^2 i},\ldots,\bar{a}^{\alpha^{o(\alpha)-1}i}),$$
and $a^lb^2 \in A_{2}\backslash\{\bar{a}^0b^2\}$, $\bar{a}^jb \in B\backslash\{\bar{a}^0b,\bar{a}^pb\}$ are respectively in the cycles
	$$ (\bar{a}^lb^2,\sigma_{\alpha,0}(\bar{a}^lb^2),\sigma_{\alpha,0}^2(\bar{a}^lb^2),\ldots,\sigma_{\alpha,0}^{o(\alpha)-1}(\bar{a}^lb^2)) = (\bar{a}^lb^2,\bar{a}^{\alpha l}b^2,\bar{a}^{\alpha^2 l}b^2,\ldots,\bar{a}^{\alpha^{o(\alpha)-1}l}b^2) ,$$
	$$ (\bar{a}^jb,\sigma_{\alpha,0}(\bar{a}^jb),\sigma_{\alpha,0}^2(\bar{a}^jb),\ldots,\sigma_{\alpha,0}^{o(\alpha)-1}(\bar{a}^jb)) = (\bar{a}^jb,\bar{a}^{\alpha j}b,\bar{a}^{\alpha^2 j}b,\ldots,\bar{a}^{\alpha^{o(\alpha)-1}j}b).$$
Thus $\sigma_{\alpha,0}$ splits $A$ into $\frac{2(p-1)}{o(\alpha)} = 2\gcd(i_{\alpha},p-1) $ cycles of length $o(\alpha) = \frac{p-1}{\gcd(i_{\alpha},p-1)}$  and 2 cycles of length $1$,  and   splits $B$ into $2$ cycles of length $1$ and $\frac{2p-2}{o(\alpha)} = 2\gcd(i_{\alpha},p-1)$ cycles of length $o(\alpha) = \frac{p-1}{\gcd(i_{\alpha},p-1)}$.
	
	Therefore, we obtain the cycle type of $\sigma_{\alpha,0}$, as shown in \eqref{equ2:1bkeven0}.

	\noindent\emph{Case 2.2.} In this case, we study the cycle type of $\sigma_{\alpha,1}$ on $D=A\cup B$.
	
Recall that all $\sigma_{\alpha,\beta}$'s have the same cycle type in $A$. Hence, $\sigma_{\alpha,1}$ splits $A$ into $\frac{2(p-1)}{o(\alpha)} = 2\mathrm{gcd}(i_{\alpha},p-1)$ cycles of length $o(\alpha)$ and $2$ cycles of length $1$.

Note that $\bar{a}^jb \in B$ is in the cycle
	$$ (\bar{a}^jb,\sigma_{\alpha,1}(\bar{a}^jb),\ldots,\sigma_{\alpha,1}^{m-1}(\bar{a}^jb) )= (\bar{a}^jb,\bar{a}^{\alpha j+1}b,\bar{a}^{\alpha^2j+\alpha + 1}b,\ldots,\bar{a}^{\alpha^{m-1}j+\alpha^{m-2}+\cdots+1}b),$$
	where $m$ is the least positive integer such that
	$$
\alpha^mj+\alpha^{m-1}+\cdots+\alpha^2+\alpha+1 \equiv j \pmod {2p},
$$
	that is,
	$$
(\alpha^{m-1}+\alpha^{m-2}+\cdots+\alpha+1)((\alpha-1)j+1) \equiv 0 \pmod {2p}.
$$

If $(\alpha-1)j+1 \equiv p \pmod {2p}$, then
\begin{equation}\label{equ:case2.2}
	(\alpha-1)j \equiv p-1 \pmod{2p}
\end{equation}
and
$$
 \alpha^{m-1}+\alpha^{m-2}+\cdots+\alpha+1 \equiv 0 \pmod 2.
 $$
So $m=2$. Notice that $\gcd(\alpha-1,p)=1$. Then the general solution of \eqref{equ:case2.2} is given by
$$ j=  (\alpha-1)^{-1}(p-1)+tp,~~t\in \mathbb{Z},$$
where $(\alpha-1)^{-1}$ is the inverse of $(\alpha-1)$ such that $(\alpha-1)^{-1}(\alpha-1) \equiv 1\pmod{p}$. Thus, $j\in \mathbb{Z}_{2p}$ can only be
$$(\alpha-1)^{-1}(p-1)~~\text{or}~~ (\alpha-1)^{-1}(p-1) + p \pmod{2p}.$$
Note that
 \begin{align*}
	\sigma_{\alpha,1}(\bar{a}^{ (\alpha-1)^{-1}(p-1)}b) & = \bar{a}^{\alpha  (\alpha-1)^{-1}(p-1) +1}b \\
	& = \bar{a}^{ (\alpha-1)(\alpha-1)^{-1}(p-1) +(\alpha-1)^{-1}(p-1)+ 1}b \\
	 	& =\bar{a}^{ (\alpha-1)^{-1}(p-1) +p}b.
\end{align*}
Then $\bar{a}^{ (\alpha-1)^{-1}(p-1) }b$ is in the cycle $(\bar{a}^{ (\alpha-1)^{-1}(p-1) }b,\bar{a}^{ (\alpha-1)^{-1}(p-1) +p}b)$. Thus we obtain $1$ cycle of length $2$.

If  $(\alpha-1)j+1 \not\equiv p \pmod{2p}$, then $o((\alpha-1)j+1 ) = 2p $ in $\mathbb{Z}_{2p}$. Hence,
	$$ \alpha^{m-1}+\alpha^{m-2}+\cdots+\alpha+1 \equiv 0 \pmod {2p}.$$
	Notice that the least positive integer of $y$ such that
	$$ \alpha^y -1 = (\alpha-1)(\alpha^{y-1}+\alpha^{y-2}+\cdots+1) \equiv 0 \pmod {2p}$$
	is $o(\alpha)$ in the multiplicative group $\mathbb{Z}_{2p}^*$.
	Then $ m \geq o(\alpha)$. Since $\alpha \in \mathbb{Z}_{2p}^{*}$ and $\alpha -1 \neq 2p,p$, we have
	$$ p \mid (\alpha^{o(\alpha) -1}+\alpha^{o(\alpha)-2}+\cdots + 1).$$
	If $o(\alpha)$ is even, then $\alpha^{o(\alpha) -1}+\alpha^{o(\alpha)-2}+\cdots + 1 $ is even. Hence $m=o(\alpha)$. Thus, $\sigma_{\alpha,1}$ splits $B$ into $\frac{2p-2}{o(\alpha)} = 2\gcd(i_{\alpha},p-1)$ cycles of length $o(\alpha)$ and one cycle of length $2$. If $o(\alpha)$ is odd, then $\alpha^{o(\alpha) -1}+\alpha^{o(\alpha)-2}+\cdots + 1  = p$. Since $\alpha^{o(\alpha)}+1 = 2$,
	$$ 2(\alpha^{o(\alpha) -1}+\alpha^{o(\alpha)-2}+\cdots + 1 ) = \alpha^{2o(\alpha)-1} + \alpha^{2o(\alpha)-2} + \cdots + 1 = 2p.$$
	So $m=2o(\alpha)$. Thus $\sigma_{\alpha,1}$ splits $B$ into $\frac{2p-2}{2o(\alpha)} = \gcd(i_{\alpha},p-1)$ cycles of length $2o(\alpha)$ and one cycle of length $2$ for odd $o(\alpha)$.
	
	Therefore, we obtain the cycle type of $\sigma_{\alpha,1}$, as shown in \eqref{equ2:1bkodd1}.

	\noindent\emph{Case 2.3.} In this case, we consider the cycle type of $\tau_{\gamma,0}$ on $D=A\cup B$.
	
If $\gamma = \alpha$, then $ \tau_{\gamma,0}$ has the same cycle type as $\sigma_{\alpha,0}$ in $A_{1}\cup \{a^p\}$. Hence, $\tau_{\gamma,0}$ splits $A_{1}\cup \{a^p\}$ into $\frac{p-1}{o(\gamma)} = \gcd(i_{\gamma},p-1)$ cycles of length $o(\gamma) = \frac{p-1}{\gcd(i_{\gamma},p-1)}$ and one cycle of length $1$.

Recall that $\tau_{\gamma,0}(\bar{a}^lb^2) = \bar{a}^{\gamma l+p}b^2,~\forall~\bar{a}^lb^2 \in A_{2}$. If $l=0$, then
	$$
\bar{a}^0b^2 = \{b^2,a^pb^2\},~~\tau_{\gamma,0}(\bar{a}^0b^2) = \bar{a}^0b^2.
$$
So $(\bar{a}^0b^2)$ is a cycle of length $1$.

Notice that $\bar{a}^lb^2\in A_{2}\backslash\{\bar{a}^0b^2\}$  is in the cycle
	\begin{equation*}
		(\bar{a}^lb^2, \tau_{\gamma,0}(\bar{a}^lb^2),\ldots,\tau_{\gamma,0}^{m-1}(\bar{a}^lb^2)) = (\bar{a}^lb^2,\bar{a}^{\gamma l + p}b^2,\ldots , \bar{a}^{\gamma^{m-1}l+\gamma^{m-2}p+\cdots+p}b^2),
	\end{equation*}
	where $m$ is the least positive integer such that
	\begin{equation}\label{equ:case2.3}
		\gamma^{m}l+\gamma^{m-1}p+\cdots+\gamma p +p
		\equiv l \pmod {2p},
	\end{equation}
	i.e.,
	$$
 (\gamma^{m-1}+\gamma^{m-2}+\cdots+\gamma+1)((\gamma-1)l+p) \equiv 0 \pmod {2p}.
$$

	Since $l\neq 0,p$, and $1 \neq \gamma\in \mathbb{Z}_{2p}^{*}$ and $\gamma$ is odd, we have
	$$(\gamma-1)l + p \equiv 1 \pmod 2 ~~\text{and}~~ (\gamma-1)l + p \not\equiv p \pmod {2p} .$$
	Then
	$$ \gamma^{m-1}+\gamma^{m-2}+\cdots + 1 \equiv 0  \pmod {2p}.$$
Thus $m$ must be even.
Note also that
	$$
\gamma^m -1 = (\gamma-1)(\gamma^{m-1}+\gamma^{m-2}+\cdots+1) \equiv 0 \pmod {2p}.
$$
Thus  $o(\gamma) \mid m$.
	
If $o(\gamma)$ is even, then $m = o(\gamma)$. Thus, $\tau_{\gamma,0}$ splits $A_{2}$ into $\frac{p-1}{o(\gamma)} = \gcd(i_{\gamma},p-1)$ cycles of length $o(\gamma)$ and one cycle of length $1$. If $o(\gamma)$ is odd, then $m = 2o(\gamma)$. Thus, $\tau_{\gamma,0}$ splits $A_{2}$ into $\frac{p-1}{2o(\gamma)} =  \frac{\gcd(i_{\gamma},p-1)}{2}$ cycles of length $2o(\gamma)$ and one cycle of length $1$.

	Recall that $\tau_{\gamma,0}(\bar{a}^jb) = \bar{a}^{\gamma j+p}b,~\forall~\bar{a}^jb \in B$. If $j = 0$ or $j=p$, then
	$$\tau_{\gamma,0}(\bar{a}^0b) = \bar{a}^pb ~~\text{and }~~\tau_{\gamma,0}(\bar{a}^pb) = \bar{a}^0b.$$
	So $(\bar{a}^0b,\bar{a}^pb)$ is a  cycle of length $2$.
	
	Notice that $\bar{a}^jb \in B\backslash\{\bar{a}^0b,\bar{a}^pb\}$ is in the cycle
	$$ (\bar{a}^jb,\tau_{\gamma,0}(\bar{a}^jb),\ldots,\tau_{\gamma,0}^{m-1}(\bar{a}^jb) )=
	(\bar{a}^jb, \bar{a}^{\gamma j+p}b , \ldots, \bar{a}^{\gamma^{{m-1}}j+\gamma^{m-2}p+\cdots+p}b),$$
	where $m$ is the least positive integer such that
	$$ \gamma^mj+\gamma^{m-1}p+\cdots+\gamma p + p \equiv j \pmod {2p},$$
	$j\neq 0,p$, and $1\neq \gamma \in \mathbb{Z}_{2p}^{*}$.
	
	Similar to the discussion on \eqref{equ:case2.3}, one can verify that: if $o(\gamma)$ is even, then $\tau_{\gamma,0}$ splits $B$ into $\frac{2p-2}{o(\gamma)} =  2\gcd(i_{\gamma},p-1)$ cycles of length $o(\gamma)$ and 1 cycle of length $2$; if $o(\gamma)$ is odd,  then $\tau_{\gamma,0}$ splits $B$ into $\frac{2p-2}{2o(\gamma)} =  \gcd(i_{\gamma},p-1)$ cycles of length $2o(\gamma)$ and 1 cycle of length $2$.
	
	Therefore, we obtain the cycle type of $\tau_{\gamma,0}$, as shown in \eqref{equ2:2bkeven0}.

	\noindent\emph{Case 2.4.} In this case, we study the cycle type of $\tau_{\gamma,1}$ on $D=A\cup B$.
	
Recall that all $\tau_{\gamma,\delta}$'s have the same cycle type in $A$. Then $\tau_{\gamma,1}$ has the same cycle type as $\tau_{\gamma,0}$ in $A$.
So, if $o(\gamma)$ is even, then $\tau_{\gamma,1}$ splits $A$  into $ 2\gcd(i_{\gamma},p-1)$ cycles of length $o(\gamma)$ and $2$ cycles of length $1$; if $o(\gamma)$ is odd, then $\tau_{\gamma,1}$ splits $A$  into $ \gcd(i_{\gamma},p-1)$ cycles of length $o(\gamma)$, $\frac{\gcd(i_\gamma,p-1)}{2}$ cycles of length $2o(\gamma)$ and $2$ cycles of length $1$.

	Recall that  $\tau_{\gamma,1}(\bar{a}^jb) =  \bar{a}^{\gamma j +1+p}b,~\forall~ \bar{a}^jb\in B$. Note that $\bar{a}^jb \in B$ is in the cycle
	$$ (\bar{a}^jb,\tau_{\gamma,1}(\bar{a}^jb),\ldots,\tau_{\gamma,1}^{m-1}(\bar{a}^jb) )= (\bar{a}^jb,\bar{a}^{\gamma j+1+p}b,\bar{a}^{\gamma^2j+\gamma + \gamma p + 1+p}b,\ldots,\bar{a}^{\gamma^{m-1}j+\gamma^{m-2}(1+p)+\cdots+1+p}b), $$
	where $m$ is the least positive integer such that
	$$ \gamma^mj + (\gamma^{m-1}+\gamma^{m-2}+\cdots+\gamma+1)(1+p) \equiv j \pmod {2p},$$
	i.e.,
	$$ (\gamma^{m-1}+\gamma^{m-2}+\cdots+\gamma+1)((\gamma-1)j+(1+p)) \equiv 0 \pmod {2p}.$$

	If $(\gamma-1)j+1+p \equiv 0 \pmod {2p}$, then
	\begin{equation}\label{equ:case2.4}
		(\gamma-1)j \equiv p-1\pmod{2p}
	\end{equation}
	 and
	$$ (\gamma^{m-1}+\gamma^{m-2}+\cdots+\gamma+1)((\gamma-1)j+(1+p)) \equiv 0 \pmod {2p}. $$
So $m=1$. 
	Notice that $\gcd(\gamma-1,p)=1$. Then the general solution of \eqref{equ:case2.4} is given by
	$$ j =  (\gamma-1)^{-1}(p-1)+tp,~~t\in \mathbb{Z},$$
	where $(\gamma-1)^{-1}$ is the inverse of $(\gamma-1)$ such that $(\gamma-1)^{-1}(\gamma-1) \equiv 1\pmod{p}$. Thus $j\in \mathbb{Z}_{2p}$ can only be
	$$ (\gamma-1)^{-1}(p-1)~~\text{or}~~ (\gamma-1)^{-1}(p-1)+ p \pmod{2p}.$$
Hence we obtain $2$ cycles of length $1$.

	If $(\gamma-1)j+1+p \not\equiv 0 \pmod {2p}$, then $(\gamma-1)j+1+p\equiv 0 \pmod {2}$. Thus
	$$
\gamma^{m-1}+\gamma^{m-2}+\cdots+\gamma+1 \equiv 0 \pmod { p},
$$
that is,
	$$
\gamma^m -1 = (\gamma-1)(\gamma^{m-1}+\gamma^{m-2}+\cdots+1)  \equiv 0  \pmod {p}.
$$
Then $o'(\gamma) \mid m$, where $o'(\gamma)$ denotes the order of $\gamma$ in $\mathbb{Z}_{p}^{*}$ such that $\gamma^{o'(\gamma)}\equiv 1\pmod{p}$. Hence $m= o'(\gamma)$.

Notice that $p\mid 2p$ and  $\gamma^{o(\gamma)}\equiv 1\pmod{2p}~~\text{for all $\gamma \in \mathbb{Z}_{2p}^*\backslash\{1\}$}$.
Then
	$$\gamma^{o(\gamma)} \equiv 1\pmod{p}~~\text{and}~~o'(\gamma) \mid o(\gamma).$$
	Since $\gamma \in \mathbb{Z}_{2p}^*\backslash\{1\}$ is odd and $\gcd(2,p)=1$, it follows that $$\gamma^{o'(\gamma)}\equiv 1\pmod{2}~~\text{and}~~\gamma^{o'(\gamma)}\equiv 1 \pmod{2p}.$$
	So, $o(\gamma) \mid o'(\gamma)$. Thus $o(\gamma) = o'(\gamma)$ for all $\gamma \in \mathbb{Z}_{2p}^*\backslash\{1\}$.
	
	Hence
	$ m = o(\gamma)$, and $\tau_{\gamma,1}$ splits $B$ into $\frac{2p-2}{o(\gamma)} = 2\gcd(i_{\gamma},p-1)$ cycles of length $o(\gamma)$ and $2$ cycles of length $1$.
	
	Therefore, we obtain the cycle type of $\tau_{\gamma,1}$, as shown in \eqref{equ2:2bkodd1}.
	\qed\end{proof}
	\begin{lemma}\label{lem:graphcycleindex}
		Let $D=A\cup B $. Then the cycle index of $\Aut(T_{8p})$ acting on $D$ is given by
		\begin{align*}
			& \mathcal{I}(\Aut(T_{8p}),D)  \\[0.3cm]
			=& \frac{1}{4p}\left(-x_{1}^{4p}+x_{1}^{2p}(-x_{2}^{p}+x_{p}^2+x_{2p}) -x_{1}^{p+1}x_{2}^{\frac{3p-1}{2}}+x_{2}^{\frac{p-1}{2}}(-x_{1}^{3p+1}+x_{1}^{p+1}x_{p}^2+x_{1}^{p+1}x_{2p})\right) \\[0.3cm]
&+ \frac{1}{4(p-1)}x_{1}^{4}\sum\limits_{\substack{d \mid (p-1)}}\varphi(d)x_{d}^{\frac{4(p-1)}{d}}
			+ \frac{1}{4(p-1)}(2x_{1}^2x_{2}+x_{1}^4)\sum_{\substack{d \mid p-1 \\d\text{~is~even}}}\varphi(d)x_{d}^{\frac{4(p-1)}{d}}
			\\[0.3cm]
&+ \frac{1}{4(p-1)}x_{1}^2x_{2}\sum_{\substack{d \mid (p-1)\\{d \text{~is~odd}}}}\varphi(d)\left(x_{d}^{\frac{2(p-1)}{d}}x_{2d}^{\frac{p-1}{d}} + x_{d}^{\frac{p-1}{d}}x_{2d}^{\frac{3(p-1)}{2d}}\right)
			+ \frac{1}{4(p-1)}x_{1}^4\sum_{\substack{d \mid p-1 \\ {d}\text{~is~odd}}}\varphi(d)x_{d}^{\frac{3(p-1)}{d}}x_{2d}^{\frac{(p-1)}{2d}},
		\end{align*}
		where $\varphi(\cdot)$ denotes the Euler's totient function.
	\end{lemma}

	\begin{proof}
By Lemma \ref{lem:graphcycletype}, the cycle index of $\Aut(T_{8p})$ acting on $D=A\cup B$ is given by
		\begin{align*}
			&\mathcal{I}\left(\Aut(T_{8p}), D\right)\\
=&\frac{1}{|\Aut(T_{8 p})|} \left( \sum\limits_{{\sigma_{\alpha,\beta}} \in \Aut(T_{8p})} x_1^{b_1(\sigma_{\alpha,\beta})}  \cdots x_{4p}^{b_{4p}(\sigma_{\alpha,\beta})}
			+ \sum\limits_{{\tau_{\gamma,\delta}} \in \Aut(T_{8p})} x_1^{b_1(\tau_{\gamma,\delta})} \cdots x_{4p}^{b_{4p}(\tau_{\gamma,\delta})} \right) \\[0.3cm]
=&\frac{1}{2\cdot 2p \cdot (p-1)}\left(
			\sum\limits_{\alpha \in \mathbb{Z}_{2p}^{*}}\sum\limits_{\beta \in \mathbb{Z}_{2p}}
			x_1^{b_1(\sigma_{\alpha,\beta})}  \cdots x_{2p}^{b_{2p}(\sigma_{\alpha,\beta})} +
			\sum\limits_{\gamma \in \mathbb{Z}_{2p}^*}\sum\limits_{\delta \in \mathbb{Z}_{2p}}
			x_1^{b_1(\tau_{\gamma,\delta})} \cdots x_{2p}^{b_{2p}(\tau_{\gamma,\delta})}
			\right)\\[0.3cm]
=&\frac{1}{4p(p-1)}\left(
			\sum\limits_{\beta \in \mathbb{Z}_{2p}}
			x_{1}^{b_{1}(\sigma_{1,\beta})}\cdots  x_{2p}^{b_{2p}(\sigma_{1,\beta})}+ \sum\limits_{\delta \in \mathbb{Z}_{2p}}
			x_{1}^{b_{1}(\tau_{1,\delta})}\cdots  x_{2p}^{b_{2p}(\tau_{1,\delta})}
			\right. \\[0.3cm]
			&  +
			\sum\limits_{\alpha \in \mathbb{Z}_{2p}^*\backslash \{1\}} \left(
			\sum\limits_{\beta \in 2\mathbb{Z}_{2p}}
			x_1^{b_1(\sigma_{\alpha,\beta})}  \cdots x_{2p}^{b_{2p}(\sigma_{\alpha,\beta})}
			+ \sum\limits_{\beta \in 2\mathbb{Z}_{2p}+1}
			x_1^{b_1(\sigma_{\alpha,\beta})}  \cdots x_{2p}^{b_{2p}(\sigma_{\alpha,\beta})}
			\right)  \\[0.3cm]
			& \left. +
			\sum\limits_{\gamma \in \mathbb{Z}_{2p}^*\backslash\{1\}}
			\left(\sum\limits_{\delta \in 2\mathbb{Z}_{2p}}
			x_1^{b_1(\tau_{\gamma,\delta})}  \cdots x_{2p}^{b_{2p}(\tau_{\gamma,\delta})} +
			\sum\limits_{\delta \in 2\mathbb{Z}_{2p}+1}
			x_1^{b_1(\tau_{\gamma,\delta})}  \cdots x_{2p}^{b_{2p}(\tau_{\gamma,\delta})} \right)
			\right)\\[0.3cm]
=& \frac{1}{4p(p-1)} \left(
			x_{1}^{4p}+x_{1}^{2p}x_{2}^{p} + (p-1)x_{1}^{2p}x_{p}^2 + (p-1)x_{1}^{2p}x_{2p}
			\right) \\[0.3cm]
			&+ \frac{1}{4p(p-1)} \left(x_{1}^{p+1}x_{2}^{\frac{3p-1}{2}}+x_{1}^{3p+1}x_{2}^{\frac{p-1}{2}}+(p-1)x_{1}^{p+1}x_{2}^{\frac{p-1}{2}}x_{p}^2+(p-1)x_{1}^{p+1}x_{2}^{\frac{p-1}{2}}x_{2p}\right)
			\\[0.3cm]
			&+ \frac{1}{4p(p-1)}\left(p\sum\limits_{\alpha \in \mathbb{Z}_{2p}^*\backslash \{1\}}x_{1}^{b_{1}(\sigma_{\alpha,0})}\cdots x_{2p}^{b_{2p}(\sigma_{\alpha,0})}
			+
			p\sum\limits_{\alpha \in \mathbb{Z}_{2p}^*\backslash \{1\}}x_{1}^{b_{1}(\sigma_{\alpha,1})}\cdots x_{2p}^{b_{2p}(\sigma_{\alpha,1})}\right.  \\[0.3cm]
			&+\left.p\sum\limits_{\gamma \in \mathbb{Z}_{2p}^*\backslash \{1\}}x_{1}^{b_{1}(\tau_{\gamma,0})}\cdots x_{2p}^{b_{2p}(\tau_{\gamma,0})}
			+
			p\sum\limits_{\gamma \in \mathbb{Z}_{2p}^*\backslash \{1\}}x_{1}^{b_{1}(\tau_{\gamma,1})}\cdots x_{2p}^{b_{2p}(\tau_{\gamma,1})}
			\right) \\[0.3cm]
=& \frac{1}{4p(p-1)} \left(\left(x_{1}^{4p}+x_{1}^{2p}x_{2}^{p} + (p-1)x_{1}^{2p}x_{p}^2 + (p-1)x_{1}^{2p}x_{2p}
			\right) + \left( x_{1}^{p+1}x_{2}^{\frac{3p-1}{2}}+x_{1}^{3p+1}x_{2}^{\frac{p-1}{2}}\right.\right. \\[0.3cm]
			& \left.\left.+(p-1)x_{1}^{p+1}x_{2}^{\frac{p-1}{2}}x_{p}^2 +(p-1)x_{1}^{p+1}x_{2}^{\frac{p-1}{2}}x_{2p}\right)\right)
			  +  \frac{1}{4(p-1)} x_{1}^{4}\sum\limits_{\alpha=z^{i_{\alpha}} \in \mathbb{Z}_{2p}^{*} \backslash \{1\}}x_{\frac{p-1}{\gcd(i_{\alpha},p-1)}}^{4\gcd(i_{\alpha},p-1)} 	\\[0.3cm]
			&+
			 \frac{1}{4(p-1)} x_{1}^2x_{2} \left(  \sum_{\substack{\alpha=z^{i_{\alpha}} \in \mathbb{Z}_{2p}^{*} \backslash \{1\}\\\frac{p-1}{\gcd(i_\alpha, p-1)} \text{~is~even}}}x_{\frac{p-1}{\gcd(i_\alpha, p-1)}}^{4\gcd(i_\alpha,p-1)}
			+
			\sum_{\substack{\alpha=z^{i_{\alpha}} \in \mathbb{Z}_{2p}^{*} \backslash \{1\}\\\frac{p-1}{\gcd(i_\alpha, p-1)} \text{~is~odd}}}x_{\frac{p-1}{\gcd(i_\alpha, p-1)}}^{2\gcd(i_\alpha,p-1)}x_{\frac{2p-2}{\gcd(i_\alpha, p-1)}}^{\gcd(i_\alpha,p-1)}
			\right) \\[0.3cm]
			&+  \frac{1}{4(p-1)} x_{1}^2x_{2} \left(  \sum_{\substack{\gamma=z^{i_{\gamma}} \in \mathbb{Z}_{2p}^{*} \backslash \{1\}\\\frac{p-1}{\gcd(i_\gamma, p-1)} \text{~is~even}}}x_{\frac{p-1}{\gcd(i_\gamma, p-1)}}^{4\gcd(i_\gamma,p-1)}
			+
			\sum_{\substack{\gamma=z^{i_{\gamma}} \in \mathbb{Z}_{2p}^{*} \backslash \{1\}\\\frac{p-1}{\gcd(i_\gamma, p-1)} \text{~is~odd}}}x_{\frac{p-1}{\gcd(i_\gamma, p-1)}}^{\gcd(i_\gamma,p-1)}x_{\frac{2p-2}{\gcd(i_\gamma, p-1)}}^{\frac{3\gcd(i_\gamma,p-1)}{2}}
			\right) \\[0.3cm]
			&+
			 \frac{1}{4(p-1)} x_{1}^4 \left(
			\sum_{\substack{\gamma=z^{i_{\gamma}} \in \mathbb{Z}_{2p}^{*} \backslash \{1\}\\\frac{p-1}{\gcd(i_\gamma, p-1)} \text{~is~even}}}x_{\frac{p-1}{\gcd(i_\gamma, p-1)}}^{4\gcd(i_\gamma,p-1)}+
			\sum_{\substack{\gamma=z^{i_{\gamma}} \in \mathbb{Z}_{2p}^{*} \backslash \{1\}\\\frac{p-1}{\gcd(i_\gamma, p-1)} \text{~is~odd}}}x_{\frac{p-1}{\gcd(i_\gamma, p-1)}}^{3\gcd(i_\gamma,p-1)}x_{\frac{2p-2}{\gcd(i_\gamma, p-1)}}^{\frac{\gcd(i_\gamma,p-1)}{2}}
			\right)   \\[0.3cm]
=& \frac{1}{4p(p-1)} \left(\left(x_{1}^{4p}+x_{1}^{2p}x_{2}^{p} + (p-1)x_{1}^{2p}x_{p}^2 + (p-1)x_{1}^{2p}x_{2p}
			\right) + \left( x_{1}^{p+1}x_{2}^{\frac{3p-1}{2}}+x_{1}^{3p+1}x_{2}^{\frac{p-1}{2}}\right.\right. \\[0.3cm]
			& \left.\left.+(p-1)x_{1}^{p+1}x_{2}^{\frac{p-1}{2}}x_{p}^2 +(p-1)x_{1}^{p+1}x_{2}^{\frac{p-1}{2}}x_{2p}\right)\right)
			  +  \frac{1}{4(p-1)} x_{1}^{4}\sum\limits_{i_\alpha\in \mathbb{Z}_{p-1}\backslash\{0\}}x_{\frac{p-1}{\gcd(i_{\alpha},p-1)}}^{4\gcd(i_{\alpha},p-1)} 	\\[0.3cm]
			& +\frac{1}{4(p-1)}x_{1}^2x_{2} \left(  \sum_{\substack{i_\alpha\in \mathbb{Z}_{p-1}\backslash\{0\}\\\frac{p-1}{\gcd(i_\alpha, p-1)} \text{~is~even}}}x_{\frac{p-1}{\gcd(i_\alpha, p-1)}}^{4\gcd(i_\alpha,p-1)}
			+
			\sum_{\substack{i_\alpha\in \mathbb{Z}_{p-1}\backslash\{0\}\\\frac{p-1}{\gcd(i_\alpha, p-1)} \text{~is~odd}}}x_{\frac{p-1}{\gcd(i_\alpha, p-1)}}^{2\gcd(i_\alpha,p-1)}x_{\frac{2p-2}{\gcd(i_\alpha, p-1)}}^{\gcd(i_\alpha,p-1)}
			\right) \\[0.3cm]
			&+\frac{1}{4(p-1)}x_{1}^2x_{2} \left(  \sum_{\substack{i_{\gamma} \in \mathbb{Z}_{p-1}\backslash\{0\}\\\frac{p-1}{\gcd(i_\gamma, p-1)} \text{~is~even}}}x_{\frac{p-1}{\gcd(i_\gamma, p-1)}}^{4\gcd(i_\gamma,p-1)}
			+
			\sum_{\substack{i_{\gamma} \in \mathbb{Z}_{p-1}\backslash\{0\}\\\frac{p-1}{\gcd(i_\gamma, p-1)} \text{~is~odd}}}x_{\frac{p-1}{\gcd(i_\gamma, p-1)}}^{\gcd(i_\gamma,p-1)}x_{\frac{2p-2}{\gcd(i_\gamma, p-1)}}^{\frac{3\gcd(i_\gamma,p-1)}{2}}
			\right) \\[0.3cm]
			&+\frac{1}{4(p-1)}x_{1}^4 \left(
			\sum_{\substack{i_{\gamma} \in \mathbb{Z}_{p-1}\backslash\{0\}\\\frac{p-1}{\gcd(i_\gamma, p-1)} \text{~is~even}}}x_{\frac{p-1}{\gcd(i_\gamma, p-1)}}^{4\gcd(i_\gamma,p-1)}+
			\sum_{\substack{i_{\gamma} \in \mathbb{Z}_{p-1}\backslash\{0\}\\\frac{p-1}{\gcd(i_\gamma, p-1)} \text{~is~odd}}}x_{\frac{p-1}{\gcd(i_\gamma, p-1)}}^{3\gcd(i_\gamma,p-1)}x_{\frac{2p-2}{\gcd(i_\gamma, p-1)}}^{\frac{\gcd(i_\gamma,p-1)}{2}}
			\right)  \\[0.3cm]
			=& \frac{1}{4p(p-1)}\left(\left(
			x_{1}^{4p}+x_{1}^{2p}x_{2}^{p} + (p-1)x_{1}^{2p}x_{p}^2 + (p-1)x_{1}^{2p}x_{2p}
			\right)
			+ \left(x_{1}^{p+1}x_{2}^{\frac{3p-1}{2}}
			+x_{1}^{3p+1}x_{2}^{\frac{p-1}{2}} \right. \right.	\\[0.3cm]
&\left. \left.+(p-1)x_{1}^{p+1}x_{2}^{\frac{p-1}{2}}x_{p}^2 +(p-1)x_{1}^{p+1}x_{2}^{\frac{p-1}{2}}x_{2p}\right)\right)
			+ \frac{1}{4(p-1)}x_{1}^{4}\sum\limits_{\substack{d \mid (p-1)\\{d\neq 1}}}\varphi(d)x_{d}^{\frac{4(p-1)}{d}} 	\\[0.3cm]
			&+\frac{1}{4(p-1)}x_{1}^2x_{2} \left(
			\sum_{\substack{d \mid p-1 \\d \text{~is~even}}}\varphi(d)x_{d}^{\frac{4(p-1)}{d}}+\sum_{\substack{d \mid p-1 \\ {d\neq 1,~d} \text{~is~odd}}}\varphi(d)x_{d}^{\frac{2(p-1)}{d}}x_{2d}^{\frac{(p-1)}{d}}
			\right)  \\[0.3cm]	
				&+\frac{1}{4(p-1)}x_{1}^2x_{2} \left(
			\sum_{\substack{d \mid p-1 \\d \text{~is~even}}}\varphi(d)x_{d}^{\frac{4(p-1)}{d}}+\sum_{\substack{d \mid p-1 \\ {d\neq 1,~d} \text{~is~odd}}}\varphi(d)x_{d}^{\frac{(p-1)}{d}}x_{2d}^{\frac{3(p-1)}{2d}}
			\right)  \\[0.3cm]	
				&+\frac{1}{4(p-1)}x_{1}^4 \left(
			\sum_{\substack{d \mid p-1 \\d \text{~is~even}}}\varphi(d)x_{d}^{\frac{4(p-1)}{d}}+\sum_{\substack{d \mid p-1 \\ {d\neq 1,~d} \text{~is~odd}}}\varphi(d)x_{d}^{\frac{3(p-1)}{d}}x_{2d}^{\frac{(p-1)}{2d}}
			\right)  \\[0.3cm]	
=& \frac{1}{4p(p-1)}\left(\left(
			(1-p)x_{1}^{4p}+(1-p)x_{1}^{2p}x_{2}^{p} + (p-1)x_{1}^{2p}x_{p}^2 + (p-1)x_{1}^{2p}x_{2p}
			\right) + \left((1-p)x_{1}^{p+1}x_{2}^{\frac{3p-1}{2}} \right. \right. 	\\[0.3cm]
			&\left.\left.
			+(1-p)x_{1}^{3p+1}x_{2}^{\frac{p-1}{2}}  +(p-1)x_{1}^{p+1}x_{2}^{\frac{p-1}{2}}x_{p}^2 +(p-1)x_{1}^{p+1}x_{2}^{\frac{p-1}{2}}x_{2p}\right)\right)\\[0.3cm]
			&+\frac{1}{4(p-1)}x_{1}^{4}\sum\limits_{d \mid (p-1)}\varphi(d)x_{d}^{\frac{4(p-1)}{d}}  	
+\frac{1}{4(p-1)}
			(2x_{1}^2x_{2}+x_{1}^4) \sum_{\substack{d \mid p-1 \\d \text{~is~even}}}\varphi(d)x_{d}^{\frac{4(p-1)}{d}} \\[0.3cm]	
&+ \frac{1}{4(p-1)}x_{1}^2x_{2}\left(
			\sum_{\substack{d \mid p-1 \\d \text{~is~odd}}}\varphi(d)x_{d}^{\frac{2(p-1)}{d}}x_{2d}^{\frac{(p-1)}{d}} +
			\sum_{\substack{d \mid p-1 \\ d \text{~is~odd}}}\varphi(d)x_{d}^{\frac{(p-1)}{d}}x_{2d}^{\frac{3(p-1)}{2d}}
			\right)  \\[0.3cm]	
			&+\frac{1}{4(p-1)}x_{1}^4\sum_{\substack{d \mid p-1 \\ d \text{~is~odd}}}\varphi(d)x_{d}^{\frac{3(p-1)}{d}}x_{2d}^{\frac{(p-1)}{2d}} \\[0.3cm]	
=& \frac{1}{4p}\left(-x_{1}^{4p}+x_{1}^{2p}(-x_{2}^{p}+x_{p}^2+x_{2p}) -x_{1}^{p+1}x_{2}^{\frac{3p-1}{2}}+x_{2}^{\frac{p-1}{2}}(-x_{1}^{3p+1}+x_{1}^{p+1}x_{p}^2+x_{1}^{p+1}x_{2p})\right) \\[0.3cm]
&+ \frac{1}{4(p-1)}x_{1}^{4}\sum\limits_{\substack{d \mid (p-1)}}\varphi(d)x_{d}^{\frac{4(p-1)}{d}}
			+ \frac{1}{4(p-1)}(2x_{1}^2x_{2}+x_{1}^4)\sum_{\substack{d \mid p-1 \\d\text{~is~even}}}\varphi(d)x_{d}^{\frac{4(p-1)}{d}}
			\\[0.3cm]
&+ \frac{1}{4(p-1)}x_{1}^2x_{2}\sum_{\substack{d \mid (p-1)\\{d \text{~is~odd}}}}\varphi(d)\left(x_{d}^{\frac{2(p-1)}{d}}x_{2d}^{\frac{p-1}{d}} + x_{d}^{\frac{p-1}{d}}x_{2d}^{\frac{3(p-1)}{2d}}\right)
			+ \frac{1}{4(p-1)}x_{1}^4\sum_{\substack{d \mid p-1 \\ {d}\text{~is~odd}}}\varphi(d)x_{d}^{\frac{3(p-1)}{d}}x_{2d}^{\frac{(p-1)}{2d}},
		\end{align*}
		where $\varphi(\cdot)$ denotes the Euler's totient function.	
	\qed\end{proof}
	By Lemmas \ref{Lem:Polya} and  \ref{lem:graphcycleindex}, we get the number of Cayley graphs over $T_{8p}$ up to isomorphism immediately.
	\begin{theorem}\label{thm:graphnumber}
		Let $p$ be an odd prime. Then the number of Cayley graphs over $T_{8p}$ up to isomorphism is equal to
		\begin{align}\label{equ:number}
			\nonumber\mathcal{N}
			=& \frac{1}{4p}\left(-2^{4p}+2^{2p}(-2^{p}+6)-2^{\frac{5p+1}{2}}+2^{\frac{p-1}{2}}(-2^{3p+1}+2^{p+3}+2^{p+2})\right) + \frac{4}{p-1}\sum\limits_{\substack{d \mid (p-1)}}\varphi(d)2^{\frac{4(p-1)}{d}}  \\[0.3cm]
			+&
			\frac{2^3}{p-1}\sum_{\substack{d \mid p-1 \\d \text{~is~even}}}\varphi(d)2^{\frac{4(p-1)}{d}}
			+ \frac{2}{p-1}\sum_{\substack{d \mid (p-1)\\{d \text{~is~odd}}}}\varphi(d)\left(2^{\frac{3(p-1)}{d}}+2^{\frac{5(p-1)}{2d}}\right)
			+ \frac{4}{p-1}\sum_{\substack{d \mid p-1 \\ {d} \text{~is~odd}}}\varphi(d)2^{\frac{7(p-1)}{2d}},
		\end{align}
		where $\varphi(\cdot)$ is the Euler's totient function.
	\end{theorem}

	In \cite{Mishna2000}, Mishna calculated the number of the circulant graphs of order $p$~($p$ prime) up to isomorphism. With the similar proof of \cite[Theorem 2.11]{Mishna2000}, we get the number of circulant graphs of order $2p$~($p$ prime).

	\begin{lemma}\label{lem:circulant2p}
		Let $p$ be an odd prime. Then the number of circulant graphs  of order $2p$ up to isomorphism is given by
		\begin{equation}\label{equ:circulant2p}
			\mathcal{N}_c=\frac{2}{p-1} \sum_{d \mid(p-1)} \varphi(d) 2^{\frac{p-1}{d}},
		\end{equation}
		where $\varphi(\cdot)$ is the Euler's totient function.
	\end{lemma}
	Recall that a Cayley graph $\text{Cay}(T_{8p},S)$ is connected if and only if $\langle S \rangle  = T_{8p}$. Thus, for $S \subseteq D = A \cup B = A_{1}\cup A_{2} \cup \{a^p\} \cup B $,  $\Cay(T_{8p},S)$ is disconnected if and only if
	$$  S \subseteq \langle a \rangle \cup \langle a \rangle b^2 \backslash \{e\} ~~ \text{or} ~~ S = \{a^jb,a^{p+j}b^3\}$$
	$$ \text{or} ~~S = \{a^jb,a^{p+j}b^3,a^p\} ~~ \text{or}~~
	S = \{a^jb,a^{p+j}b^3,b^2,a^pb^2\} ~~ \text{or}~~
	S = \{a^jb,a^{p+j}b^3,b^2,a^pb^2,a^p\}$$
	$$\text{or}~~S =  \{a^jb,a^{p+j}b^3,a^{j+p}b,a^jb^3 \} ~~\text{or}~~ S =  \{a^jb,a^{p+j}b^3,a^{j+p}b,a^jb^3,a^p \}$$
	$$\text{or}~~S =  \{a^jb,a^{p+j}b^3,a^{j+p}b,a^jb^3,b^2,a^pb^2 \} ~~\text{or}~~ S =  \{a^jb,a^{p+j}b^3,a^{j+p}b,a^jb^3,b^2,a^pb^2,a^p \}$$
	for $j\in Z_{2p}$. Note that
	$$ \Cay(T_{8p}, \{a^jb,a^{p+j}b^3 \}) \cong \Cay(T_{8p}, \{b,a^pb^3 \}),$$
	$$\Cay(T_{8p},\{a^jb,a^{p+j}b^3,a^p\}) \cong \Cay(T_{8p},\{b,a^pb^3,a^p\}),$$
	$$ \Cay(T_{8p},\{a^jb,a^{p+j}b^3,b^2,a^pb^2\}) \cong \Cay(T_{8p},\{b,a^{p}b^3,b^2,a^pb^2\}),$$
	$$ \Cay(T_{8p},\{a^jb,a^{p+j}b^3,b^2,a^pb^2,a^p\}) \cong \Cay(T_{8p},\{b,a^{p}b^3,b^2,a^pb^2,a^p\}),$$
	$$ \Cay(T_{8p},\{a^jb,a^{p+j}b^3,a^{j+p}b,a^jb^3\}) \cong
	\Cay(T_{8p},\{b,a^{p}b^3,a^{p}b,b^3\}),$$
	$$ \Cay(T_{8p},\{a^jb,a^{p+j}b^3,a^{j+p}b,a^jb^3,a^p\}) \cong
	\Cay(T_{8p},\{b,a^{p}b^3,a^{p}b,b^3,a^p\}),$$
	$$ \Cay(T_{8p}, \{a^jb,a^{p+j}b^3,a^{j+p}b,a^jb^3,b^2,a^pb^2\}) \cong
	\Cay(T_{8p}, \{b,a^{p}b^3,a^{p}b,b^3,b^2,a^pb^2\}),$$
	$$ \Cay(T_{8p}, \{a^jb,a^{p+j}b^3,a^{j+p}b,a^jb^3,b^2,a^pb^2,a^p\}) \cong
	\Cay(T_{8p}, \{b,a^{p}b^3,a^{p}b,b^3,b^2,a^pb^2,a^p\}),$$
	for each $j$, since
	$$\sigma_{1,j}(b) = a^jb,~\sigma_{1,j}(b^2) = b^2,~\sigma_{1,j}(b^3) = a^jb^3,~\sigma_{1,j}(a^p) = a^p,$$
	$$\sigma_{1,j}(a^pb) = a^{p+j}b,~\sigma_{1,j}(a^pb^2) = a^pb^2~\text{and}~\sigma_{1,j}(a^pb^3) = a^{p+j}b^3.$$
	

	\begin{theorem}\label{thm:connectgraphs}
		Let $p$ be an odd prime. Then the number of connected Cayley graphs over $T_{8p}$ up to isomorphism is equal to
		$$ \mathcal{N}' = \mathcal{N} - \mathcal{N}_{c}^2-8,$$
		where $\mathcal{N}$ and $\mathcal{N}_{c}$ are shown in  \eqref{equ:number} and \eqref{equ:circulant2p}.
	\end{theorem}

\begin{proof}
Recall that $\Cay(T_{8p},S)$ is connected if and only if $\langle S \rangle = T_{8p}$. Thus, if  $S \subseteq A_{1}\cup\{a^p\}\cup A_{2}$, then $\Cay(T_{8p},S)$ is disconnected. So, if $\emptyset \neq S \subseteq A_{1} \cup \{a^p\} = \langle a \rangle $,  by Lemma \ref{lem:circulant2p}, the number of disconnected Cayley graphs $\Cay(T_{8p},S)$ up to isomorphism is equal to $\mathcal{N}_{c}-1$. Similarly, if $\emptyset \neq S \subseteq A_{2} = \langle a \rangle b^2$, then the number of disconnected Cayley graphs $\Cay(T_{8p},S)$ up to isomorphism is also equal to $\mathcal{N}_{c}-1$. Thus, if $S \subseteq A_{1}\cup\{a^p\}\cup A_{2}$, then the number of disconnected Cayley graphs $\Cay(T_{8p},S)$ up to isomorphism is equal to
$$
1+ (\mathcal{N}_{c}-1) + (\mathcal{N}_{c}-1) + (\mathcal{N}_{c}-1)^2 =\mathcal{N}_{c}^2.
$$
Note also that there are $8$ more disconnected Cayley graphs $\Cay(T_{8p}, S)$ up to isomorphism listed above, where $S \cap B \neq \emptyset$. Hence, by Theorem \ref{thm:graphnumber}, the number of connected Cayley graphs over $T_{8p}$ up to isomorphism is equal to
$$
\mathcal{N}'= \mathcal{N} -  \mathcal{N}_{c}^2 - 8.
$$
This completes the proof.
		\qed\end{proof}

	\begin{table}[ht]
		\caption{\textbf{The Number of (Connected) Cayley Graphs over $T_{8p}$~($3 \leq  p \leq 13$)}}
		\label{table:graph}
		\centering
		\begin{tabular}{cccc}
			\hline
			$p$ & $\mathcal{N}$ & $\mathcal{N}_{c}$ & $\mathcal{N}'$   \\
			\hline
			3 & 432 & 6 & 388   \\
			\hline
			5 & 18144 & 12 & 17992\\
			\hline
			7 & 1824384 & 28 & 1823592 \\
			\hline
			11 & 41253667584 & 216 & 41253620920 \\
			\hline
			13 & 7330997009984 & 704 & 7330996514360 \\
			\hline
		\end{tabular}
	\end{table}
	
Table \ref{table:graph} lists the number of (connected) Cayley graphs over $T_{8p}$ ($p$ prime) up to isomorphism for $3 \leq p \leq 13$ by applying Theorems \ref{thm:graphnumber} and \ref{thm:connectgraphs}, where $\mathcal{N}$ and $\mathcal{N}'$ denote the number of Cayley graphs and connected Cayley graphs over $T_{8p}$ respectively,  and $\mathcal{N}_{c}$ is defined in Lemma  \ref{lem:circulant2p}.	
	
	\section{Conclusion}
	In  this paper, we enumerate the number of Cayley graphs over $T_{8p}$ ($p$ is odd prime) up to isomorphism by using the P{\'o}lya Enumeration Theorem. We obtain a formula for enumerating connected Cayley graphs over $T_{8p}$. We also list the explicit number of (connected) Cayley graphs over $T_{8p}$ for $3\leq p\leq 13$. It is known that CI-groups play an important role in enumerating Cayley graphs. Therefore, we conclude the paper by proposing the following problem: \textbf{Characterize other kinds of nonabelian CI-groups and enumerate the number of Cayley graphs over them.}

\end{document}